\documentclass[12pt]{amsart}

\usepackage{amssymb,amsthm,graphicx,caption,mwe,float,adjustbox, comment, hyperref}%

\makeatletter
\newcommand*\rel@kern[1]{\kern#1\dimexpr\macc@kerna}
\newcommand*\widebar[1]{%
  \begingroup
  \def\mathaccent##1##2{%
    \rel@kern{0.8}%
    \overline{\rel@kern{-0.8}\macc@nucleus\rel@kern{0.2}}%
    \rel@kern{-0.2}%
  }%
  \macc@depth\@ne
  \let\math@bgroup\@empty \let\math@egroup\macc@set@skewchar
  \mathsurround\z@ \frozen@everymath{\mathgroup\macc@group\relax}%
  \macc@set@skewchar\relax
  \let\mathaccentV\macc@nested@a
  \macc@nested@a\relax111{#1}%
  \endgroup
}
\makeatother

\newtheorem{thm}{Theorem}
\newtheorem{que}{Question}
\newtheorem{proj}{Project}
\newtheorem{theorem}{Theorem}[section]

\newtheorem{proposition}[theorem]{Proposition}

\title [An answer regarding automorphisms of finite abelian groups]{An answer regarding automorphisms of finite abelian groups}
\author[McCulloch]{Ryan McCulloch}
\address{Department of Mathematics and Statistics, Binghamton University, Binghamton, NY 13902; rmccullo1985@gmail.com}
\date{\today}

\begin{document}
	
\begin{abstract}
In this note we provide a negative answer to the question: ``Is it true that for every positive rational number $r$ there exists a finite abelian group $G$ such that $|\mathrm{Aut}(G)|/|G| = r$?".  We show that if $r = a/b$ is a rational number (with $a$ and $b$ coprime integers) so that $r = |\mathrm{Aut}(G)|/|G|$ for a finite abelian group $G$, then $b$ is squarefree.  We also show that no odd prime can equal $ |\mathrm{Aut}(G)|/|G|$ for a finite abelian group $G$.
\end{abstract}


\maketitle

\section{Introduction}

In \cite{numaut} it is shown that the set of all rational numbers $r$ so that $r = |\mathrm{Aut}(G)|/|G|$ for a finite abelian group $G$ is dense in the interval $[0, \infty)$.  In the same paper, the following open problem is given.

\begin{que}
Is it true that for every positive rational number $r$ there exists a finite (abelian) group $G$ such that $|\mathrm{Aut}(G)|/|G| = r$?
\end{que}

The question also appears as \#21.97 the ``Kourovka notebook'', see \cite{notebook}.  We answer the question in the negative for finite abelian groups.  The question is still open (as far as we know) for the class of all finite groups.  The question has recently been answered in the positive for graphs, monoids, partial groups,
and posets, see \cite{monoid}.

Our main results are as follows.

\begin{thm}\label{Intro thm 1}
If $r = a/b = |\mathrm{Aut}(G)|/|G|$ for a finite abelian group $G$, where $a$ and $b$ are coprime integers, then $b$ is squarefree.
\end{thm}

\begin{thm}\label{Intro thm 2}
Every power of $2$ can equal $ |\mathrm{Aut}(G)|/|G|$ for a finite abelian group $G$.
\end{thm}

\begin{thm}\label{Intro thm 3}
No odd prime can equal $ |\mathrm{Aut}(G)|/|G|$ for a finite abelian group $G$.
\end{thm}

The results follow by analyzing a formula for $|\mathrm{Aut}(G)|$ when $G$ is a finite abelian $p$-group, and extending to finite abelian groups via the fundamental theorem of abelian groups.

\section{Proofs of Results}

The description of $\mathrm{Aut}(G)$ when $G$ is a finite abelian group has been known since at least 1907, see \cite{oldaut}.  By the fundamental theorem of abelian groups, and since $\mathrm{Aut}(G \times H) = \mathrm{Aut}(G) \times \mathrm{Aut}(H)$ when $G$ and $H$ have coprime order, it suffices to understand $\mathrm{Aut}(G)$ when $G$ is an abelian $p$-group.  We work with the formula for $|\mathrm{Aut}(G)|$ appearing in \cite{aut}.

\begin{proposition}\label{prop: form}
Let $G \cong \mathbb{Z}_{p^{e_1}} \times \cdots \times \mathbb{Z}_{p^{e_n}}$ be a finite abelian $p$-group, where $1 \leq e_1 \leq \dots \leq e_n$.  Then $$|\mathrm{Aut}(G)| = \prod_{k=1}^{n} (p^{d_k} - p^{k-1}) \prod_{j=1}^{n} (p^{e_j})^{n-d_j} \prod_{i=1}^{n} (p^{e_i-1})^{n-c_i+1},$$ where $d_r = \max\{s: e_s = e_r\}$ and $c_r = \min\{s: e_s = e_r\}$ for $r = 1, \dots, n$.
\end{proposition}

Theorem \ref{Intro thm 2} follows directly from Proposition \ref{prop: form} by exhibiting the required abelian groups.

\begin{proposition}
We have the following.
\begin{enumerate}
\item For each $i \geq 1$, $G \cong \mathbb{Z}_{2^{i}} \times \mathbb{Z}_{2^{i+1}}$ has $|\mathrm{Aut}(G)|/|G| = 2^{2(i-1)}$. 
\item For each $i \geq 2$, $G \cong \mathbb{Z}_{2} \times \mathbb{Z}_{2^{i}} \times \mathbb{Z}_{2^{i+1}}$ has $|\mathrm{Aut}(G)|/|G| = 2^{2i+1}$. 
\item For $G \cong \mathbb{Z}_{2} \times \mathbb{Z}_{3} \times \mathbb{Z}_{9}$, we have $|\mathrm{Aut}(G)|/|G| = 2$. 
\item For $G \cong \mathbb{Z}_{2} \times \mathbb{Z}_{5} \times \mathbb{Z}_{25}$, we have $|\mathrm{Aut}(G)|/|G| = 8$. 
\end{enumerate}
\end{proposition}

We obtain a formula for the largest $p$-power dividing $|\mathrm{Aut}(G)|$ by considering the contribution coming from each product in Proposition \ref{prop: form}.  In our formulation below, we let $k_i$ denote the multiplicity of a $\mathbb{Z}_{p^{e_i}}$ term appearing in $G$, and so note that in Proposition \ref{prop: dc}, we have $1 \leq e_1 < \dots < e_m$.

\begin{proposition}\label{prop: dc}
Let $G \cong (\mathbb{Z}_{p^{e_1}})^{k_1} \times \cdots \times (\mathbb{Z}_{p^{e_m}})^{k_m}$ be a finite abelian $p$-group, where $1 \leq e_1 < \dots < e_m$ and each $k_i \geq 1$.  Let $n = k_1 + \cdots + k_m$ and let $$d = \sum_{j=1}^{m-1} \Big( e_j k_j \big(\sum_{i=j+1}^m k_i \big) \Big)$$ and let $$c = \sum_{j=1}^{m} \Big( (e_j-1)k_j\big(\sum_{i=j}^m k_i \big) \Big).$$  Then $p^{((n-1)n/2) + d + c}$ is the largest $p$-power dividing $|\mathrm{Aut}(G)|$.
\end{proposition}

\begin{proposition}\label{prop: class}
Let $G$ be a finite abelian $p$-group.  
\begin{enumerate}
\item If $G$ is cyclic, then $|\mathrm{Aut}(G)|/|G| = (p-1)/p$.
\item If $G \cong \mathbb{Z}_{p} \times \mathbb{Z}_{p}$, then $|\mathrm{Aut}(G)|/|G| = (p-1)^2(p+1)/p$.
\item If $G \cong \mathbb{Z}_{p} \times \mathbb{Z}_{p^i}$, $i > 1$, then $|\mathrm{Aut}(G)|/|G| = (p-1)^2$.
\item If $G \cong \mathbb{Z}_{p} \times \mathbb{Z}_{p} \times \mathbb{Z}_{p}$, then $$|\mathrm{Aut}(G)|/|G| =(p-1)^3(p+1)(p^2+p+1).$$
\end{enumerate}
In all other cases, we have $|\mathrm{Aut}(G)|/|G|$ is an integer and is divisible by $p(p-1)^2$.
\end{proposition}

\begin{proof}
Items (1)--(4) are readily verified using Proposition \ref{prop: form}.  If $G$ is not cyclic, then by Proposition \ref{prop: form} we have $(p-1)^2$ divides $|\mathrm{Aut}(G)|$.  We work to show that in all cases other than items (1)--(4), we have $|\mathrm{Aut}(G)|/|G|$ is an integer and is divisible by $p$.

Use the same notation as in Proposition \ref{prop: dc} so that $|G| = p^a$ where $a = {\sum_{i=1}^m e_i k_i}$ and $n = {\sum_{i=1}^m k_i}$.  Note that $c \geq a-n$.  Also note that $(n-1)n/2 \geq n$ when $n \geq 3$ and $(n-1)n/2 = n-1$ when $n=1$ or $n=2$.  Suppose $n \geq 3$.  By Proposition \ref{prop: dc}, $|\mathrm{Aut}(G)|/|G|$ is an integer, and in this case (when $n \geq 3$), since $c + n \geq a$, we have $|\mathrm{Aut}(G)|/|G|$ coprime with $p$ implies $d=0$ and $c = a - n$.  Now $d=0$ implies $m=1$ and hence $k_1 = n \geq 3$.  Thus $c = a - n$ implies $e_1 = 1$, and we are in the situation of item (4).  Thus we can suppose $n \leq 2$.  Note that $n=1$ implies that $G$ is a cyclic, which is item (1), and so suppose $n=2$.  Hence $G \cong \mathbb{Z}_{p^{e_1}} \times \mathbb{Z}_{p^{e_2}}$ or $G \cong (\mathbb{Z}_{p^{e_1}})^2$.  In cases other than item (2) or (3), we have $e_1 > 1$.  If $G \cong \mathbb{Z}_{p^{e_1}} \times \mathbb{Z}_{p^{e_2}}$, then $m = 2$, and it follows that $d = e_1 \geq 2$.  Hence $c + d + n \geq a+2$ which implies $c + d + (n-1)n/2 \geq a+1$, and we have $|\mathrm{Aut}(G)|/|G|$ is an integer and is divisible by $p$.  If $G \cong (\mathbb{Z}_{p^{e_1}})^2$, then $m=1$ and $k_1=2$, and $c = 4(e_1-1)$ and $a-n = 2(e_1-1)$.  So $c \geq a-n + 2$, and hence $c + (n-1)n/2 \geq a+1$, and thus $|\mathrm{Aut}(G)|/|G|$ is an integer and is divisible by $p$. 
\end{proof}

Theorem \ref{Intro thm 1} follows.  Write $G \cong G_1 \times \cdots \times G_n$ for $G$ a finite abelian group with each $G_i$ a $p_i$-group, $p_i$ a prime, and $p_i \neq p_j$ for $i \neq j$.  We have $$|\mathrm{Aut}(G)|/|G| = \prod_{i=1}^n |\mathrm{Aut}(G_i)|/|G_i|.$$  Write each $|\mathrm{Aut}(G_i)|/|G_i|$ as $a_i/b_i$, where $a_i$ and $b_i$ are coprime integers (so $b_i=1$ if $|\mathrm{Aut}(G_i)|/|G_i|$ is an integer).  By Proposition \ref{prop: class}, either $b_i=1$ or $b_i=p_i$.  Hence $|\mathrm{Aut}(G)|/|G| = (\prod_{i=1}^n a_i)/(\prod_{i=1}^n b_i)$ where $\prod_{i=1}^n b_i$ is a product of distinct primes.  If we write $|\mathrm{Aut}(G)|/|G| = a/b$ where $a$ and $b$ are coprime integers, then $b$ is a product of distinct primes (or $b=1$), and hence $b$ is squarefree.

We are ready to prove Theorem \ref{Intro thm 3}.

\begin{proposition}
No odd prime can equal $|\mathrm{Aut}(G)|/|G|$ when $G$ is a finite abelian group.
\end{proposition}

\begin{proof}
Via the fundamental theorem of abelian groups, write $G = G_1 \times \cdots \times G_n$ where each direct factor $G_i$ is a $p_i$-group, $p_i$ a prime, and $p_1 < \dots < p_n$.  Suppose $p$ is an odd prime and $p = |\mathrm{Aut}(G)|/|G|$.  One writes each $|\mathrm{Aut}(G_i)|/|G_i| = a_i/b_i$ where $a_i$ and $b_i$ are coprime integers, and by Proposition \ref{prop: class} we have $b_i=1$ or $b_i = p_i$. We argue that $n \leq 2$ and $G_1$ is a $2$-group.  Otherwise, if two odd order groups $G_i$ and $G_j$, $i \neq j$, appear as direct factors of $G$, then by Proposition \ref{prop: class}, $(p_i-1)$ divides $a_i$ and $(p_j-1)$ divides $a_j$, and so $4$ divides $a_i a_j$. As $p =|\mathrm{Aut}(G)|/|G| = (\prod_{i=1}^n a_i)/(\prod_{i=1}^n b_i)$, and as $4$ does not divide $\prod_{i=1}^n b_i$, a product of distinct primes, we conclude that $p=2$, a contradiction.  We now argue that $n=1$ and hence $G$ is a $2$-group.  Suppose otherwise.  If $|\mathrm{Aut}(G_2)|/|G_2|$ is an integer, then by Proposition \ref{prop: class}, $(p_2-1)^2$ divides $a_2$ and hence $4$ divides $a_2$, and we conclude that $p=2$, a contradiction.  So $b_2 = p_2$.  If $a_2 \neq p_2-1$, then $(p_2-1)^2$ divides $a_2$, again leading to a contradiction.  So $a_2/b_2 = (p_2-1)/p_2$.  As $p > 2$, and since $2$ divides $a_2$, we must have $b_1 = 2$.  Since $b_2 = p_2$, we cannot have $a_1 = 1$, and it follows by Proposition \ref{prop: class} that $a_1 = 3$.  Now $|\mathrm{Aut}(G)|/|G| = 3(p_2-1)/(2p_2)$ is a prime, and $p_2$ and $p_2-1$ are coprime, and so it follows that $p_2 - 1 = 2$ and $p=p_2=3$.  But then $\mathrm{Aut}(G)|/|G| = 1$, a contradiction.  So $n=1$ and $G$ is a $2$-group.  Now, by Proposition \ref{prop: class}, we see that $|\mathrm{Aut}(G)|/|G|$ is never an odd prime, yielding the final contradiction.
\end{proof}

We close with a few projects related to this research.

\begin{proj}
Completely classify all of the rational numbers that appear as $|\mathrm{Aut}(G)|/|G|$ when $G$ is a finite abelian group.
\end{proj}

\begin{proj}
Describe the rational numbers that appear as $|\mathrm{Aut}(G)|/|G|$ when $G$ is a finite nilpotent group.  Do all of them appear?
\end{proj}

\end{document}